\documentclass[10pt]{article}
\usepackage{graphicx}
\usepackage{tikz}
\usepackage[margin=1in]{geometry}
\usepackage{amsmath, amssymb, amsfonts, amsthm}
\newtheorem{lem}{\noindent {\bf Lemma}}[section]
\newtheorem{prop}{\noindent {\bf Proposition}}[section]
\newtheorem{coro}{\noindent {\bf Corollary}}[section]
\newtheorem{thm}{\noindent {\bf Theorem}}[section]

\newcounter{remark}
\newenvironment{remark}{\smallskip\noindent {\bf Remark \arabic{section}.\arabic{remark}.}}
{\addtocounter{remark}{1}\par}
\newcounter{example}
\newenvironment{exa}{\smallskip\noindent {\bf Example \arabic{section}.\arabic{example}.}}
{\addtocounter{example}{1}\par}
\catcode`@=11 
\date{}
\newcounter{defi}
\newenvironment{defi}{
\smallskip \noindent
{\bf
  Definition \arabic{section}.\arabic{defi}.
}}{\addtocounter{defi}{1}\par}
\newcommand{\ZZ}{\mathbb Z}
\newcommand{\NN}{\mathbb N}
\newcommand{\RR}{\mathbb R}

\title{\bf Classification of metric spaces with infinite asymptotic dimension}
\author{\large  Yan Wu$^\ast$\qquad Jingming Zhu$^{\ast\ast}$
\footnote{
College of Mathematics Physics and Information Engineering, Jiaxing University, Jiaxing , 314001, P.R.China.
$^\ast$ E-mail: yanwu@mail.zjxu.edu.cn $^\ast\ast$ E-mail: 122411741@qq.com }}
\date{}

\begin{document}
\maketitle
\begin{center}
\begin{minipage}{0.9\textwidth}
\noindent{\bf Abstract.}
We introduce a geometric property complementary-finite asymptotic dimension (coasdim). Similar with asymptotic dimension, we prove the corresponding coarse invariant theorem, union theorem and Hurewicz-type theorem. Moreover, we show that coasdim$(X)\leq fin+k$ implies trasdim$(X)\leq \omega+k-1$ and transfinite asymptotic dimension of the shift union sh$\bigcup\prod_{i=1}^{\infty}i\mathbb{Z}$ is no more than $\omega+1$. i.e., trasdim(sh$\bigcup\prod_{i=1}^{\infty}i\mathbb{Z})\leq\omega+1$.


\end{minipage}
\end{center}
\footnote{
This research was supported by
the National Natural Science Foundation of China under Grant (No.11301224,11326104)

}
\begin{section}{Introduction}\

In coarse geometry, asymptotic dimension of a metric space is an important concept which was defined by Gromov for studying asymptotic invariants of discrete groups \cite{Gromov}. This dimension can be considered as an asymptotic analogue of the Lebesgue covering dimension. As a large scale analogue of W.E. Haver¡¯s property C in dimension theory, A. Dranishnikov introduced the notion of asymptotic property C in \cite{Mari2014}. It is well known that every metric space with finite asymptotic dimension has asymptotic property C \cite{Dra00}. But the inverse is not true, which means that there exists some metric space $X$ with
infinite asymptotic dimension and asymptotic property C. Therefore how to classify the metric spaces with infinite asymptotic dimension into smaller categories becomes an interesting problem.  T. Radul defined trasfinite asymptotic dimension (trasdim) which can be viewed as transfinite extension for asymptotic dimension and gave examples of metric spaces with trasdim$=\omega$ and with trasdim$=\infty$ (see \cite{Radul2010}). He also proved trasdim$(X)<\infty$ if and only if $X$ has asymptotic property C for every metric space $X$. But whether there is a metric space $X$ with $\omega<trasdim(X)<\infty$ is still unknown so far.

In this paper, we introduce another approach to classify the metric spaces with infinite asymptotic dimension, which is called complementary-finite asymptotic dimension (coasdim), and give some examples of metric spaces with different complementary-finite asymptotic dimensions. Moreover, we prove some properties of complementary-finite asymptotic dimension and show that coasdim$(X)\leq fin+k$ implies trasdim$(X)\leq \omega+k-1$ for every metric space $X$.

The paper is organized as follows: In Section 2, we recall some definitions and properties of transfinite asymptotic dimension. In
Section 3, we introduce complementary-finite asymptotic dimension and give some examples of metric spaces with different complementary-finite asymptotic dimensions.  Besides, we prove some properties of complementary-finite asymptotic dimension like coarse invariant theorem, Union Theorem and Hurewicz-type theorem. Finally, we investigate the relationship between complementary-finite asymptotic dimension and transfinite asymptotic dimension. In Section 4, we give an example of a metric space $X$ with coasdim$(X)\leq fin+1$ and  coasdim$(X\times X)< fin+fin$ is not true, which shows that complementary-finite asymptotic dimension is not stable under direct product. However, trasdim$(X)=$trasdim$(X\times X)=\omega$. In Section 5, we define the shift union of $\prod_{i=1}^{\infty}i\mathbb{Z}$ by sh$\bigcup\prod_{i=1}^{\infty}i\mathbb{Z}$ and  prove that transfinite asymptotic dimension of sh$\bigcup\prod_{i=1}^{\infty}i\mathbb{Z}$ is no more than $\omega+1$. Finally, we give a negative answer to the Question 7.1 raised in \cite{BellNag2017}.
\end{section}

\begin{section}{Preliminaries}\

Our terminology concerning the asymptotic dimension follows from \cite{Bell2011} and for undefined terminology we refer to \cite{Radul2010}. Let~$(X, d)$ be a metric space and $U,V\subseteq X$, let
\[
\text{diam}~ U=\text{sup}\{d(x,y): x,y\in U\}
\text{   and   }
d(U,V)=\text{inf}\{d(x,y): x\in U,y\in V\}.
\]

Let $R>0$ and $\mathcal{U}$ be a family of subsets of $X$, $\mathcal{U}$ is said to be \emph{$R$-bounded} if
\[
\text{diam}~\mathcal{U}\stackrel{\bigtriangleup}{=}\text{sup}\{\text{diam}~ U: U\in \mathcal{U}\}\leq R.
\]
$\mathcal{U}$ is said to be \emph{uniformly bounded} if there exists $R>0$
such that $\mathcal{U}$ is $R$-bounded.\\
Let $r>0$, $\mathcal{U}$ is said to be\emph{ $r$-disjoint} if
\[
d(U,V)\geq r~~~~~\text{for every}~ U,V\in \mathcal{U}\text{~and~}U\neq V.
\]

A metric space $X$ is said to have \emph{finite asymptotic dimension} if
there is an $n\in\NN$, such that for every $r>0$,
there exists a sequence of uniformly bounded families
$\{\mathcal{U}_{i}\}_{i=0}^{n}$ of subsets of $X$
such that the family
$\bigcup_{i=0}^{n}\mathcal{U}_{i}$ covers $X$ and each $\mathcal{U}_{i}$
is $r$-disjoint for $i=0,1,\cdots,n$. In this case, we say that asdim$X\leq n$.

We say that asdim$X= n$ if  asdim$X\leq n$ and asdim$X\leq n-1$ is not true.\

A metric space $X$ is said to have \emph{asymptotic property C} if for every sequence $R_0<R_1<...$ of positive real numbers, there exist an $n\in\mathbb{N}$ and uniformly bounded families $\mathcal{U}_0,...,\mathcal{U}_n$ of subsets of $X$ such that each $\mathcal{U}_i$ is $R_i$-disjoint for $i=0,1,\cdots,n$ and the family $\bigcup_{i=0}^n\mathcal{U}_i$ covers $X$.\

Let $X$ and $Y$ be metric spaces. A map $f :X \to Y$ between metric spaces is a \emph{coarse embedding} if there exist
non-decreasing functions $\rho_1$ and $\rho_2$, $\rho_i:\mathbb{R}^+\cup\{0\}\to\mathbb{R}^+\cup\{0\}$ such that $\rho_i(x)\to+\infty$ as $x\to+\infty$ for $i=1,2$ and for every $x,x'\in X$
$$\rho_1(d_X(x,x'))\leq d_Y(f(x),f(x'))\leq\rho_2(d_X(x,x')).$$

It is easy to see that finite asymptotic dimension is a coarsely invariant property of metric spaces.
\begin{lem}\rm(Coarse invariance, see~\cite{Bell2011})
\label{lem:Coarseasdim}
Let $X$ and $Y$ be two metric spaces, let $\phi$ be a coarse embedding from $X$ to $Y$. If asdim$Y\leq n$, then asdim$X\leq n$.
\end{lem}

In \cite{Radul2010}, T. Radul generalized asymptotic dimension of a metric space $X$ to transfinite asymptotic dimension which is denoted by trasdim$(X)$.

Let $Fin~\mathbb{N}$ denote the collection of all finite, nonempty
subsets of $\mathbb{N}$ and let $ M \subset Fin~\mathbb{N}$. For $\sigma\in \{\varnothing\}\bigcup Fin~\mathbb{N}$, let
$$M^{\sigma} = \{\tau\in Fin\mathbb{N} ~|~ \tau \cup \sigma \in M \text{ and } \tau \cap \sigma = \varnothing\}.$$

Let $M^a$ abbreviate $M^{\{a\}}$ for $a \in \NN$. Define ordinal number Ord$M$ inductively as follows:
\begin{eqnarray*}
\text{Ord}M = 0 &\Leftrightarrow& M = \varnothing,\\
\text{Ord}M \leq \alpha &\Leftrightarrow& \forall~ a\in \mathbb{N}, ~\text{Ord}M^a < \alpha,\\
\text{Ord}M = \alpha &\Leftrightarrow& \text{Ord}M \leq \alpha \text{ and } \text{Ord}M < \alpha \text{ is not true},\\
\text{Ord}M = \infty &\Leftrightarrow& \text{Ord}M \leq\alpha \text{ is not true for every ordinal number } \alpha.
\end{eqnarray*}

\begin{lem}\rm(\cite{Borst1988})
\label{lem:OrdM}
Let $M\subseteq Fin\mathbb{N}$ and $n\in\mathbb{N}$, then Ord$M\leq n$ if and only if for each $\sigma\in M$, $|\sigma|\leq n.$
\end{lem}

\begin{lem}\rm(\cite{Borst1988})
\label{lem:Ord}
If $M_{1}\subseteq M_{2}\subseteq Fin\mathbb{N}$, then Ord$M_{1}\leq$ Ord$M_{2}.$
\end{lem}

Given a metric space $(X, d)$, define the following collection:
\[
\begin{split}
A(X, d) = \{\sigma \in Fin\mathbb{N} |~&\text{ there are no uniformly bounded families } \mathcal{U}_i  \text{ for } i \in \sigma
 \\& \text{ such that each } \mathcal{U}_i
\text{ is } i\text{-disjoint and }\bigcup_{i\in\sigma}\mathcal{U}_i \text{~covers~} X\}.
\end{split}\]

The \emph{transfinite asymptotic dimension} of $X$ is defined as trasdim$X$=Ord$A(X, d)$.

\begin{remark}
\begin{itemize}
\item There are some equivalent definitions of transfinite asymptotic dimension in \cite{Satkiewicz}.
\item  It is not difficult to see that transfinite asymptotic dimension is a generalization of finite asymptotic dimension by Lemma \ref{lem:OrdM}.
That is, $trasdim X\leq n$ if and only if $asdim X\leq n$ for each $n\in\NN$.
\end{itemize}
\end{remark}

A subset $M\subseteq Fin\mathbb{N}$ is said to be \emph{inclusive} if for every $\sigma, \tau\in Fin\mathbb{N}$ such that $\tau\subseteq\sigma$, $\sigma\in M$
implies $\tau\in M$. The following facts are easy to see.
\begin{itemize}
\item $A(X,d)$ is inclusive.
\item If $n<m$, then $A(X,d)^{n}\subseteq A(X,d)^{m}$. It follows that Ord$A(X,d)^{n}\leq$ Ord$A(X,d)^{m}$ by Lemma \ref{lem:Ord}.
\end{itemize}

\begin{lem}\rm(see \cite{Radul2010})
\label{lem:E}
Let $X$ be a metric space, $X$ has asymptotic property $C$ if and only if trasdim$X<\infty$.
\end{lem}

\begin{prop}\label{{iffforomega}}
Given a metric space $X$ with asdim$(X) = \infty$,  let $k\in \mathbb{N}$, the following are equivalent:
\begin{itemize}
\item\rm(1) trasdim$(X) \leq \omega+k$;
\item\rm(2) For every $n\in\NN$, there exists $m(n)\in\NN,$ such that for every~ $d\geq n+k$, there are uniformly bounded families $\mathcal{U}_{-k},\mathcal{U}_{-k+1},\cdots,\mathcal{U}_{m(n)}$ satisfying $\mathcal{U}_i$ is $n$-disjoint for $i=-k,\cdots, 0$, $\mathcal{U}_j$ is $d$-disjoint for $j=1,2,\cdots, m(n)$ and
$\bigcup_{i=-k}^{m(n)}\mathcal{U}_i$ covers $X$. Moreover, $m(n)\rightarrow \infty$ as $n\rightarrow \infty$;
\item\rm(3) For every $n\in\NN$, there exists $m(n)\in\NN,$ such that for every~ $d>0$, there are uniformly bounded families $\mathcal{U}_{-k},\mathcal{U}_{-k+1},\cdots,\mathcal{U}_{m(n)}$ satisfying $\mathcal{U}_i$ is $n$-disjoint for $i=-k,\cdots, 0$, $\mathcal{U}_j$ is $d$-disjoint for $j=1,2,\cdots, m(n)$ and
$\bigcup_{i=-k}^{m(n)}\mathcal{U}_i$ covers $X$. Moreover, $m(n)\rightarrow \infty$ as $n\rightarrow \infty$.
\end{itemize}
\end{prop}

\begin{proof}
(1)~$\Rightarrow$ (2): Assume that trasdim$(X) \leq \omega+k$, i.e.,
Ord$A(X,d) \leq \omega+k$.
By definition,  for every $n\in\NN$, let $\tau=\{n,n+1,\cdots, n+k\}\in\text{Fin}\NN$, $\text{Ord}A(X,d)^\tau<\omega.$
It follows that there exists $ m(n)\in\NN$ such that $\text{Ord}A(X,d)^\tau=m(n)-1<m(n).$
By Lemma \ref{lem:OrdM}, for each $\sigma\in A(X,d)^\tau$, $|\sigma|<m(n).$
So for every $d\geq n+k$, let $\alpha=\{d+1,d+2,\cdots,d+m(n)\}$, then $\tau\sqcup\alpha\notin A(X,d)$.
Then there are uniformly bounded families $\mathcal{U}_{-k}, \mathcal{U}_{-k+1}, ...,\mathcal{U}_{m(n)}$ such that $\mathcal{U}_i$ is $n$-disjoint for $i\in\{-k,\cdots, 0\}$, $\mathcal{U}_j$ is $d$-disjoint for $j\in\{1,2,\cdots, m(n)\}$ and
$\bigcup_{i=-k}^{m(n)}\mathcal{U}_i$ covers $X$. Since asdim$(X) = \infty$, we obtain that $m(n)\rightarrow \infty$ whenever $n\rightarrow \infty$.

(2)~$\Rightarrow$ (1): By definition, if suffices to show that for each
$\tau=\{a_{0},a_{1},\cdots,a_{k}\}\in\text{Fin}\NN$,
 $\text{Ord}A(X,d)^\tau<m(n).$ i.e.,
 \[ \text{if~} \sigma\in A(X,d)^\tau, \text{then~} |\sigma|<m(n).
 \]
Now we will prove that if $\sigma\cap\tau=\emptyset$ and $|\sigma|\geq m(n)$, then $\sigma\sqcup\tau\notin A(X,d)$.
 Since $A(X,d)$ is inclusive, it suffices
for us to show $\sigma_{1}\sqcup\tau\notin A(X,d)$ for some $\sigma_{1}=\{b_{1},b_{2},\cdots,b_{m(n)}\}\subseteq\sigma$. Let $n=\max\{a_{0},a_{1},\cdots,a_{k}\}$ and let $d=\max\{b_{1},b_{2},\cdots,b_{m(n)}\}+n+k$,
by condition (2), there are uniformly bounded families $\mathcal{U}_{-k},\mathcal{U}_{-k+1},\cdots,\mathcal{U}_{m(n)}$ satisfying $\mathcal{U}_i$ is $n$-disjoint for $i=-k,\cdots, 0$, $\mathcal{U}_j$ is $d$-disjoint for $j=1,2,\cdots, m(n)$ and
$\bigcup_{i=-k}^{m(n)}\mathcal{U}_i$ covers $X$.
Therefore, $\mathcal{U}_i$ is $a_{i}$-disjoint for $i=-k,\cdots, 0$ and $\mathcal{U}_j$ is $b_{j}$-disjoint for $j=1,2,\cdots, m(n)$.
That is, $\sigma_{1}\sqcup\tau\notin A(X,d)$.

(2)~$\Leftrightarrow$ (3): Obvious.

\end{proof}

We end this section with some technical lemmas similar with Lemma 3 in \cite{Radul2010} with similar proof.
\begin{lem}\rm(see~\cite{Radul2010})
\label{lem:dsup}
 Let $k,n\in\NN$ with $n>2$ and $R\in\RR$ with $R>n$, then there are no $n$-disjoint and $R$-bounded families $\mathcal{V}_1,\mathcal{V}_2,...,\mathcal{V}_k$ in $\ZZ^k$ with maximum metric $d_{\text{max}}$ such that $\bigcup_{i=1}^k\mathcal{V}_i$ covers $[-R,R]^k\bigcap\ZZ^k$.
\end{lem}
\begin{proof}
Suppose that there exist $n$-disjoint and $R$-bounded families $\mathcal{V}_1,\mathcal{V}_2,...,\mathcal{V}_k$ in $\ZZ^k$ such that $\bigcup_{i=1}^k\mathcal{V}_i$ covers $[-R,R]^k\bigcap\ZZ^k$. Let $\mathcal{V}=\bigcup_{i=1}^{k}\mathcal{V}_i$ and $\mathcal{U}=\{N_{1}(V)\bigcap[-R,R]^k~|~V\in\mathcal{V}\}$, where $N_1(V)=\{x\in \RR^k~|~d(x,V)\leq1\}$. Then $\mathcal{U}$ is a finite closed cover of cube $[-R,R]^k$ such that
\begin{itemize}
\item\rm(1) no member of which meets two opposite faces of $[-R,R]^k$;
\item\rm(2) each subfamily of $\mathcal{U}$ containing $k+1$ distinct elements of $\mathcal{U}$ has empty intersection.
\end{itemize}
 So we obtain
a contradiction with the Lebesgue's Covering Theorem(see \cite{Engelking}, Theorem 1.8.20).

\end{proof}

\begin{coro}\rm
\label{lem:dsup2}
Let $k,n,l\in\mathbb{N}$ and $n>2l$, let $R\in\RR$ and $R>n$, there are no $n$-disjoint and $R$-bounded families of subsets $\mathcal{V}_1,...,\mathcal{V}_k$ in $(l\mathbb{Z})^k$ with  maximum metric $d_{\text{max}}$ such that $\bigcup_{i=1}^{k}\mathcal{V}_i$ covers $[-R,R]^k\bigcap (l\mathbb{Z})^k$.
\end{coro}
\end{section}

\begin{section}{Cofinite asymptotic dimension}\

In this section, we define complementary-finite asymptotic dimension (coasdim), which is a generalization of finite asymptotic dimension.

\begin{defi}
\begin{itemize}
\item A metric space $X$ is said to have \emph{complementary-finite asymptotic dimension} if there is a $k\in \mathbb{N}$, such that
for every $n\in\NN,$ there exist $R>0$, $m\in\NN$ and $n$-disjoint, $R$-bounded families $\mathcal{U}_1,...,\mathcal{U}_k$ of subsets of $X$ such that
\[\text{ asdim}(X\setminus \bigcup(\bigcup_{i=1}^k\mathcal{U}_i))\leq m.
\]
In this case, we say that coasdim$(X)\leq fin+k$.

\item We say that coasdim$(X)= fin+1$ if coasdim$(X)\leq fin+1$ and asdim$(X)=\infty$.

\item Let $k\in\NN$ and $k>1$, we say that coasdim$(X)= fin+k$ if coasdim$(X)\leq fin+k$ and coasdim$(X)\leq fin+k-1$ is not true.

\item We say that coasdim$(X)< fin+fin$ if coasdim$(X)\leq fin+k$ for some $k\in\mathbb{N}$.
\end{itemize}
\end{defi}

\begin{exa}
Let $k\ZZ=\{kj|~j\in\ZZ\}$ and $X=\bigcup_{i=1}^{\infty}(i\ZZ)^i$. Let $a=(a_1,a_2,...,a_l)\in (l\ZZ)^l$ and $b=(b_1,b_2,...,b_k)\in (k\ZZ)^k$ with $l\leq k$. Let $a'=(a_1,...,a_l,0,...,0)\in \ZZ^k$. Put $c=0$ if $l=k$ and $c=l+(l+1)+...+(k-1)$ if $l<k$. Define a metric $d$ on $X$ by
$$d(a,b)=\max\{d_{\text{max}}(a',b),c\}\text{ where $d_{\text{max}}$ is the maximum metric in }\ZZ^k.$$

T. Radul proved that trasdim$(X)=\omega$ (see \cite{Radul2010}). We can obtain that coasdim$(X)=fin+1$.

Indeed, it is easy to see that asdim$(X)=\infty$. For every $n\in\NN$, let $\mathcal{U}_1=\{\{x\}|x\in\bigcup_{i=n+1}^{\infty}(i\ZZ)^i\}$. It is easy to see that $\mathcal{U}_1$ is a $n$-disjoint, uniformly bounded family and
\[
\text{asdim}(X\setminus\bigcup \mathcal{U}_1)=\text{asdim}(\bigcup_{i=1}^{n}(i\ZZ)^i)<\infty.
\]

\end{exa}

\begin{exa}
Let $k\in \ZZ^+$ and let $X=\bigcup_{i=1}^{\infty}((i\ZZ)^i\times \ZZ^k)$. Let $a=(a_1,a_2,...,a_{l_{1}},\bar{a}_1,...,\bar{a}_k)\in (l_{1}\ZZ)^{l_{1}}\times \ZZ^k$ and $b=(b_1,b_2,...,b_{l_{2}},\bar{b}_1,...,\bar{b}_k)\in (l_{2}\ZZ)^{l_{2}}\times \ZZ^k$ with $l_{1}\leq l_{2}$. Let
$$a'=(a_1,...,a_{l_{1}},0,...,0,\bar{a}_1,...,\bar{a}_k)\in \ZZ^{l_{2}}\times \ZZ^k.$$

Put $c=0$ if $l_{1}=l_{2}$ and $c=l_{1}+(l_{1}+1)+...+(l_{2}-1)$ if $l_{1}<l_{2}$. Define a metric $d$ on $X$ by
$$d(a,b)=\max\{d_{\text{max}}(a',b),c\},\text{ where $d_{\text{max}}$ is the maximum metric in }\ZZ^{l_{2}}\times \ZZ^k.$$
We will show that coasdim$(X)=fin+k+1$.

\begin{thm}\label{exampleofcoasdimfin+t}
Let $k\in \NN, k\geq1$ and let $X=(\bigcup_{i=1}^{\infty}((i\ZZ)^i\times \ZZ^k),d),$ where $d$ is the metric defined above, then coasdim$(X)=fin+k+1$.
\end{thm}

\end{exa}
\begin{proof}
$\mathbf{Claim~1:}$ coasdim$(X)\leq fin+k+1$.\

Since asdim~$\ZZ^{k}=k$,
for every $n\in \NN$ and $n>2$, there are $n$-disjoint and uniformly bounded families $\mathcal{U}_1,...,\mathcal{U}_{k+1}$ such that $\bigcup_{i=1}^{k+1}\mathcal{U}_i$ covers $\ZZ^k$. For $i=1,2,\cdots,k+1,$ let
$$\mathcal{V}_i=\{\{x\}\times U~|~x\in \bigcup_{j=n+1}^{\infty}(j\ZZ)^j \text{ and } U\in \mathcal{U}_i\}.$$

It is easy to see $\mathcal{V}_i$ is $n$-disjoint and uniformly bounded.  Moreover,
 \[\text{asdim}~(X\setminus(\bigcup(\bigcup_{i=1}^{k+1} \mathcal{V}_i)))= \text{asdim}~(\bigcup_{j=1}^{n}((j\ZZ)^j\times \ZZ^k))<\infty.\]

$\mathbf{Claim~2:}$ coasdim$(X)\leq fin+k$ is not true.\

Suppose that coasdim$(X)\leq fin+k$, then for every $n\in\NN$, there exist $R>n$, $ m\in\NN$ and  $n$-disjoint $R$-bounded families $\mathcal{U}_1,...,\mathcal{U}_k$ such that
\[\text{ asdim}(X\setminus \bigcup(\bigcup_{i=1}^k\mathcal{U}_i))\leq m.
\]

Choose $l\in\NN$ such that $l>n+R+m$ and let $\mathcal{U}_i(l)\triangleq\{U\cap((l\ZZ)^l\times \ZZ^k)|~U\in\mathcal{U}_i\}$. Then
\[\text{asdim$(((l\ZZ)^l\times \ZZ^k)\setminus (\bigcup(\bigcup_{i=1}^k\mathcal{U}_i(l))))\leq \text{ asdim}(X\setminus \bigcup(\bigcup_{i=1}^k\mathcal{U}_i))\leq m$}.
\]

Fix $x\in (l\ZZ)^l$, for $i=1,2,\cdots,k$, if $U\cap (\{x\}\times \ZZ^{k})\neq\emptyset \text{~for some~} U\in\mathcal{U}_i$,
then let $V_{U}=P_{\ZZ^{k}}(U\cap (\{x\}\times \ZZ^{k}))$, where $P_{\ZZ^{k}}(.)$ is the projection onto $\ZZ^{k}$.
Let $\mathcal{V}_i(x)=\{V_{U}~|~U\cap (\{x\}\times \ZZ^{k})\neq\emptyset\text{~and~}U\in\mathcal{U}_i\}$. If $U\cap (\{x\}\times \ZZ^{k})=\emptyset$ for every $U\in\mathcal{U}_i$,
then let $\mathcal{V}_i(x)=\emptyset$.
Then the family $\mathcal{V}_i(x)$ is $n$-disjoint and $R$-bounded.

Note that if $\mathcal{V}_i(x)\neq\emptyset$, then
\[\mathcal{U}_i(l)=\{\{x\}\times V|x\in (l\ZZ)^l,~V\in \mathcal{V}_i(x)\}.\]

Indeed, for any $(x,y),(x',y')\in U\in \mathcal{U}_i(l)$ with $x,x'\in (l\mathbb{Z})^l$ and $y,y'\in \mathbb{Z}^k$, if $x\neq x'\in (l\ZZ)^l$, then $d((x,y),(x',y'))\geq l>R$, which is a contradiction to the $R$-boundedness of $U$. \

Therefore
\[((l\ZZ)^l\times ([-R,R]^k\bigcap\ZZ^k))\setminus (\bigcup(\bigcup_{i=1}^k\mathcal{U}_i(l)))=\bigcup_{x\in (l\ZZ)^l}(\{x\}\times (([-R,R]^k\bigcap\ZZ^k)\setminus \bigcup(\bigcup_{i=1}^k\mathcal{V}_i(x)))).\]

By Lemma \ref{lem:dsup}, for any $x\in(l\ZZ)^l$, $(\{x\}\times([-R,R]^k\bigcap\ZZ^k)\setminus \bigcup(\bigcup_{i=1}^k\mathcal{V}_i(x)))\neq\emptyset$.
Hence there exists $(x, \Delta(x))\in (\{x\}\times([-R,R]^k\bigcap\ZZ^k)\setminus \bigcup(\bigcup_{i=1}^k\mathcal{V}_i(x)))$ such that
\[(x,\Delta(x))\in ((l\ZZ)^l\times ([-R,R]^k\bigcap\ZZ^k))\setminus (\bigcup(\bigcup_{i=1}^k\mathcal{U}_i(l)))=\bigcup_{x\in (l\ZZ)^l}(\{x\}\times (([-R,R]^k\bigcap\ZZ^k)\setminus \bigcup(\bigcup_{i=1}^k\mathcal{V}_i(x)))).\]

Define a map $\delta$ from $(l\ZZ)^l$ to $((l\ZZ)^l\times ([-R,R]^k\bigcap\ZZ^k))\setminus (\bigcup(\bigcup_{i=1}^k\mathcal{U}_i(l))) $ by
$$\delta(x)=(x,\Delta(x)).$$

Since
\[
\forall x,y\in (l\ZZ)^l, d(x,y)\leq d(\delta(x),\delta(y))\leq d(x,y)+2R,
 \]
$\delta$ is a coarse embedding from  $(l\ZZ)^l$ to $((l\ZZ)^l\times ([-R,R]^k\bigcap\ZZ^k))\setminus (\bigcup(\bigcup_{i=1}^k\mathcal{U}_i(l))).$
Therefore,
\[
m<l=\text{asdim}((l\ZZ)^l)\leq\text{asdim}(((l\ZZ)^l\times ([-R,R]^k\bigcap\ZZ^k))\setminus(\bigcup(\bigcup_{i=1}^k\mathcal{U}_i(l))))
\]
\[
\leq\text{asdim$(((l\ZZ)^l\times \ZZ^k)\setminus (\bigcup(\bigcup_{i=1}^k\mathcal{U}_i(l))))$}\leq m,
\]
which is a contradiction.

\end{proof}

\begin{remark}
It is worth to notice that coasdim$(X)\leq fin+k$ for some $k\in\mathbb{N}$ if and only if for every
$ n > 0, \exists~ m(n)\in\NN, \exists$ $n$-disjoint and uniformly bounded families of subsets $\mathcal{U}_{-1},\mathcal{U}_{-2},\cdots,\mathcal{U}_{-k}$ such that for every $d >0$,
$\exists~ d$-disjoint uniformly bounded families $\mathcal{U}_0, \mathcal{U}_1, ..., \mathcal{U}_{m(n)}$ with
$\bigcup(\bigcup_{i=-k}^{m(n)}\mathcal{U}_i)=X$.
\end{remark}

By Proposition 2.1, the difference between coasdim$(X)\leq fin+k$ and trasdim$(X)\leq\omega+k-1$ is whether the families $\mathcal{U}_{-1},\mathcal{U}_{-2},\cdots,\mathcal{U}_{-k}$ is independent of the number $d$. So it is easy to obtain the following Theorem.

\begin{thm}\label{coasdimfin+1impliestrasdimomega}
Let $X$ be a metric space , if coasdim$(X)\leq fin+k$ for some $k\in\mathbb{N}$, then trasdim$(X)\leq\omega+k-1$. Especially, $coasdim(X)=fin+1$ implies $trasdim(X)=\omega$.
\end{thm}

But the inverse is not true. Let $X=\bigcup_{i=1}^{\infty}((2^i\ZZ)^i\times \ZZ)$ be the metric space in Example 3.2,  we can obtain that coasdim$X= fin+2$, but trasdim$(X)=\omega$ by the result below.

\begin{prop}
Let $X=\bigcup_{i=1}^{\infty}((2^i\ZZ)^i\times \ZZ)$ be the metric space with the metric $d$ in Example 3.2, then trasdim$(X)=\omega$.
\end{prop}
\begin{proof}
By Proposition 2.1, it suffices to show that for every $n\in\NN$, $n>2$ and for every $r\in\NN$ and $r>n$, there are uniformly bounded families $\mathcal{U}_0, \mathcal{U}_1, \cdots, \mathcal{U}_{n+5}$ such that
$\mathcal{U}_0$ is $n$-disjoint, $\mathcal{U}_i$ is $r$-disjoint for $i=1,2,\cdots, n+5$ and $\bigcup_{i=0}^{n+5}\mathcal{U}_i$ covers $X$.
Let
$$\mathcal{U}_{n+4}=\{\{x\}\times [(2k+1)r,(2k+2)r]:x\in \bigcup_{i=r}^{\infty}(2^i\ZZ)^i,k\in\ZZ\}$$
$$\mathcal{U}_{n+5}=\{\{x\}\times [(2k)r,(2k+1)r]:x\in \bigcup_{i=r}^{\infty}(2^i\ZZ)^i,k\in\ZZ\}.$$
It is easy to see that $\mathcal{U}_{n+4}$ and $\mathcal{U}_{n+5}$ are $r$-disjoint, uniformly bounded families.  Moreover, $\mathcal{U}_{n+4}\bigcup\mathcal{U}_{n+5}$ covers  $\bigcup_{i=r}^{\infty}((2^i\ZZ)^i\times\ZZ)$.

Define a map $\varphi: (\bigcup_{i=1}^{n}((2^i\ZZ)^i\times \ZZ), d)\longrightarrow (\ZZ^{n+1}\times [0,\frac{n(n-1)}{2}],d_{\text{max}})$ as follows:\
$$\varphi(\widetilde{x})=(x_{1},x_{2},\cdots, x_{i},0,\cdots,0,x_{\ast},\frac{i(i-1)}{2}),~\forall~\widetilde{x}=(x_{1},x_{2},\cdots, x_{i},x_{\ast})\in(2^i\ZZ)^i\times\ZZ.$$
Note that $\varphi$ is an isometric map. Since asdim$(\ZZ^{n+1}\times[0, \frac{n(n-1)}{2}], d_{\text{max}})\leq n+1$, asdim$(\bigcup_{i=1}^{n}((2^i\ZZ)^i\times \ZZ))\leq n+1$ by Lemma \ref{lem:Coarseasdim}. Hence there are uniformly bounded families $\mathcal{U}_{2}, \mathcal{U}_{3},\cdots, \mathcal{U}_{n+3}$ such that each $\mathcal{U}_{i}$
is $r$-disjoint and $\bigcup_{i=2}^{n+3}\mathcal{U}_i$ covers $\bigcup_{i=1}^{n}((2^i\ZZ)^i\times\ZZ)$.

Now we are going to show that there are uniformly bounded families $\mathcal{U}_{0}$ and $\mathcal{U}_{1}$ such that $\mathcal{U}_{0}$ is $n$-disjoint,
 $\mathcal{U}_{1}$ is $r$-disjoint and $\mathcal{U}_{0}\bigcup \mathcal{U}_{1}$ covers $\bigcup_{i=n+1}^{r}((2^i\ZZ)^i\times\ZZ)$.

There is an isometric map $\psi: (\bigcup_{i=n+1}^{r}((2^i\ZZ)^i\times\ZZ), d)\longrightarrow (((2^n\ZZ)^r\times\ZZ\times [\frac{n(n+1)}{2},\frac{r(r-1)}{2}]),d_{\text{max}})$ defined as follows:
$$\psi(\widetilde{x})=(x_{1},x_{2},\cdots, x_{i},0,\cdots,0,x_*,\frac{i(i-1)}{2}),~\forall~ \widetilde{x}=(x_{1},x_{2},\cdots, x_{i},x_*)\in(2^i\ZZ)^i\times\ZZ.$$

Hence it is suffices to show there are uniformly bounded families $\mathcal{V}_{0}$ and $\mathcal{V}_{1}$ such that $\mathcal{V}_{0}$ is $n$-disjoint,
 $\mathcal{V}_{1}$ is $r$-disjoint and $\mathcal{V}_{0}\bigcup \mathcal{V}_{1}$ covers $(2^n\ZZ)^r\times\ZZ\times [\frac{n(n+1)}{2},\frac{r(r-1)}{2}]$.

For every $~x=(x_{1},x_{2},\cdots, x_{r})\in (2^n\ZZ)^r$,
let
\[
p(x)=\sum_{j=1}^{r}(\frac{x_j}{2^{r}}-\lfloor\frac{x_j}{2^{r}}\rfloor)2^{rj}.
\]
It is not difficult to see
\[
p(x)\in[0,\sum_{j=1}^{r}2^{rj}-1]\cap\ZZ.
\]
For every $~k\in\ZZ\setminus\{0\}$, let
\[
 I_{x,k}=[(k\sum_{j=1}^{r}2^{rj}+p(x))(r+n)-n, (k\sum_{j=1}^{r}2^{rj}+p(x))(r+n)]
\]
\[
 J_{x,k}=[(k\sum_{j=1}^{r}2^{rj}+p(x))(r+n), ((k+1)\sum_{j=1}^{r}2^{rj}+p(x))(r+n)-n].
\]

Then we denote
\[
U_{x,k}=\{x\}\times J_{x,k}\times [\frac{n(n+1)}{2},\frac{r(r-1)}{2}]\text{,   }~~
V_{x,k}=\{x\}\times I_{x,k}\times [\frac{n(n+1)}{2},\frac{r(r-1)}{2}]
\]
and
\[
\mathcal{V}_{0}=\{U_{x,k}~|~x\in(2^n\ZZ)^r, k\in\ZZ\}\text{, }~~\mathcal{V}_{1}=\{V_{x,k}~|~x\in(2^n\ZZ)^r, k\in\ZZ\}.
\]

It is easy to see that $\mathcal{V}_{0}$ is $n$-disjoint and uniformly bounded, $\mathcal{V}_{1}$ is uniformly bounded  and  $\mathcal{V}_{0}\bigcup \mathcal{V}_{1}$ covers $(2^n\ZZ)^r\times\ZZ\times [\frac{n(n+1)}{2},\frac{r(r-1)}{2}]$.

Now  we only need to prove that $\mathcal{V}_{1}$ is $r$-disjoint.\\
For every $V_{x,k}=\{x\}\times I_{x,k}\times [\frac{n(n+1)}{2},\frac{r(r-1)}{2}], V_{x',k'}=\{x'\}\times I_{x',k'}\times [\frac{n(n+1)}{2},\frac{r(r-1)}{2}]\in \mathcal{V}_1$ and $V_{x,k}\neq V_{x',k'}$.
\begin{itemize}
\item If $d(x,x')\geq 2^{r}$, then $d(V_{x,k},V_{x',k'})\geq 2^r\geq r$.
\item If $0<d(x,x')< 2^{r}$, then $|p(x)-p(x')|\geq 1$. Since
$$\min\{(k\sum_{j=1}^{r}2^{rj}+p(x)-k'\sum_{j=1}^{r}2^{rj}-p(x'))(r+n):d(x,x')< 2^{r}\text{ and }k,k'\in\ZZ\}\geq r+n,$$
then
$$\min\{d(I_{x,k},I_{x',k'}):d(x,x')< 2^{r}\text{ and }k,k'\in\ZZ\}\geq r$$
and hence $d(V_{x,k},V_{x',k'})\geq r$.
\item If $d(x,x')=0$, then $x=x'$. since $V_{x,k}\neq V_{x,k'}$, we have $k\neq k'$. It follows that $d(V_{x,k},V_{x,k'})=d(I_{x,k},I_{x,k'})\geq r$.

\end{itemize}
Therefore $\mathcal{V}_{1}$ is $r$-disjoint and uniformly bounded.

Let $\mathcal{U}_{0}=\{\psi^{-1}(V)~|~V\in\mathcal{V}_{0}\}$ and $\mathcal{U}_{1}=\{\psi^{-1}(V)~|~V\in\mathcal{V}_{1}\}$,
then $\mathcal{U}_{0}$ is $n$-disjoint, uniformly bounded and $\mathcal{U}_{1}$ is $r$-disjoint, uniformly bounded.
Moreover, $\mathcal{U}_{0}\cup\mathcal{U}_{1}$ covers $\bigcup_{i=n+1}^{r}((2^i\ZZ)^i\times\ZZ)$.
\end{proof}

\textbf{Question:} For $k\in\NN$ and $k>1$, does coasdim$(X)=fin+k$ imply trasdim$(X)=\omega$? \\

Similar with finite asymptotic dimension,  complementary-finite asymptotic dimension is invariant under some operations.

\begin{prop}\rm(Coarse invariant Theorem)
\label{prop:coarseinvariant}
Let $X$ and $Y$ be two metric spaces, let $\phi:X\longrightarrow Y$ be a coarse embedding from $X$ to $Y$. If  coasdim$(Y)\leq fin+k$  for some $k\in\NN$, then  coasdim$(X)\leq fin+k$.
\end{prop}
\begin{proof}
For every $r>0$, since $\phi$ is a coarse embedding, there exists $m>0$ such that for every $x,y\in X$, if $d(\phi(x),\phi(y))\geq m$,
then $d(x,y)\geq r$. Since coasdim$(Y)\leq fin+k$, there exist $R>0$ and
$~m$-disjoint $R$-bounded families $\mathcal{U}_1,...,\mathcal{U}_k $ such that asdim$(Y\backslash(\bigcup(\bigcup_{i=1}^k\mathcal{U}_i)))\leq n$ for some $n\in\NN$.
For $i=1,...,k$, let $\mathcal{W}_i=\{\phi^{-1}(U)|U\in \mathcal{U}_i\}$, then
$\mathcal{W}_i$ is $r$-disjoint. Since $\phi$ is a coarse embedding,  there exists $S>0$ such that $\mathcal{W}_i$ is $S$-bounded families of subsets in $X$.
By Lemma \ref{lem:Coarseasdim},
asdim$(X\backslash(\bigcup(\bigcup_{i=1}^k\mathcal{W}_i)))\leq n$.  Hence, coasdim$(X)\leq fin+k$.
\end{proof}

\begin{coro}\label{almostfiniteunderinclusion}
Let $Y$ be a metric space and $X\subset Y$, if coasdim$(Y)\leq fin+k$, then coasdim$(X)\leq fin+k$.
\end{coro}

\begin{defi}\label{dimuniformlyless}(\cite{Bell2011})
Let $\{X_{\alpha}\}_{\alpha}$ be a family of metric spaces, we say that asdim$(X_{\alpha})\leq n$ uniformly if
for every $r>0$, there exists $R>0$ such that for each $\alpha$, there exist $r$-disjoint and $R$-bounded families $\mathcal{U}_{\alpha}^0,  \mathcal{U}_{\alpha}^1,\cdots, \mathcal{U}_{\alpha}^n$~ of subsets of $X_{\alpha}$ satisfying
$\bigcup(\bigcup_{i=0}^n\mathcal{U}_{\alpha}^i)=X_{\alpha}$.
\end{defi}

\begin{defi}\label{almostfinitedimuniformlyless}
Let $\{X_{\alpha}\}_{\alpha}$ be a family of metric spaces, we say that coasdim~$(X_{\alpha})\leq fin+k$ uniformly if
\[
\text{for every } n>0, \text{there exist~} m=m(n)\in\NN, R=R(n)>0, \text{such that for each~} \alpha, \exists~n\text{-disjoint and $R$-bounded  } \] families $\mathcal{U}_{\alpha}^1,...,\mathcal{U}_{\alpha}^k$ of subsets of $X_{\alpha}$ with
$$ \text{ asdim}(X_{\alpha}\backslash \bigcup(\bigcup_{i=1}^k\mathcal{U}_{\alpha}^i))\leq m\text{~uniformly }.$$
\end{defi}

\begin{defi}(\cite{Bell2011})
Let $\mathcal{U}$ and $\mathcal{V}$ be families of subsets of $X$. The \emph{$r$-saturated union} of $\mathcal{V}$ with $\mathcal{U}$ is defined as
$$\mathcal{V}\cup_r\mathcal{U}=\{N_r(V;\mathcal{U})|V\in\mathcal{V}\}\cup\{U\in\mathcal{U}|d(U,\mathcal{V})>r\},$$
where $N_r(V;\mathcal{U}) = V \cup\bigcup_{d(U,V)\leq r} U$ and $d(U,\mathcal{V})>r$ means that for every $V\in\mathcal{V}$, $d(U,V)>r$.
\end{defi}

\begin{lem}\rm(Proposition 7 in \cite{Bell2011})
\label{r-saturated}
Let $\mathcal{U}$ be a $r$-disjoint and $R$-bounded family of subsets of $X$ with $R\geq r$. Let $\mathcal{V}$ be a $5R$-disjoint, $D$-bounded family of subsets of $X$. Then the family $\mathcal{V}\cup_r\mathcal{U}$ is $r$-disjoint, $D+2R+2r$-bounded.

\end{lem}

\begin{thm}\label{finiteunion}
Let $X=\bigcup_{\alpha} X_{\alpha}$ be a metric space where the family $\{X_{\alpha}\}_{\alpha}$ satisfies coasdim~$(X_{\alpha})\leq fin+k$ uniformly for some $k\in\NN$.
For every $r>0,$ if there is a $Y_r\subseteq X$ with coasdim~$(Y_r)\leq fin+k$ and
\[d(X_{\alpha}\backslash Y_r,X_{\alpha'}\backslash Y_r)\geq r \text{~~whenever~~} X_{\alpha}\neq X_{\alpha'},
\]
then coasdim~$(X)\leq fin+k$.
\end{thm}

\begin{proof}
Since coasdim~$(X_{\alpha})\leq fin+k$ uniformly, for every $n>0$, there exist $m=m(n)\in\NN$ and $R=R(n)\geq n$, such that for each $\alpha$, there are $n$-disjoint and $R$-bounded families $\mathcal{U}_{\alpha}^{-1},\mathcal{U}_{\alpha}^{-2},...,\mathcal{U}_{\alpha}^{-k}\text{ of subsets of~}X_{\alpha}$ and
\[
\text{ asdim~}(X_{\alpha}\backslash \bigcup(\bigcup_{i=-k}^{-1}\mathcal{U}_{\alpha}^i))\leq m\text{ uniformly.  }
\]

Then for every $r>n, \exists~ D=D(r)\geq r$ such that for each $\alpha$, there exist $r$-disjoint and $D$-bounded families $\mathcal{U}_{\alpha}^0,...,\mathcal{U}_{\alpha}^{m}$ such that
$\bigcup_{i=0}^{m}\mathcal{U}_{\alpha}^i$ covers $X_{\alpha}\backslash \bigcup(\bigcup_{i=-k}^{-1}\mathcal{U}_{\alpha}^i)$.
So $\bigcup_{i=-k}^{m}\mathcal{U}_{\alpha}^i$ covers $X_{\alpha}$.

Since coasdim~$Y_{r}\leq fin+k$, there exist
$5R$-disjoint, uniformly bounded families
$\mathcal{V}^{-k},...,\mathcal{V}^{-1}$  such that
\[
\text{ asdim~}(Y_{r}\backslash \bigcup(\bigcup_{i=-k}^{-1}\mathcal{V}^i))\leq l\text{ ~for some~} l\in\NN.
\]

Then there exist
$5D$-disjoint and uniformly bounded families $\mathcal{V}^{0},\mathcal{V}^{1},...,\mathcal{V}^{l} \text{ such that }
$ $\bigcup_{i=-k}^{l}\mathcal{V}^i \text{ covers } Y_{r}.$

For $i\in\{-k,\cdots,-1,0,1,\cdots,m\}$,
let $\widetilde{\mathcal{U}_{\alpha}^i}$ to be the restriction of $\mathcal{U}_{\alpha}^i$ to $X_{\alpha}\backslash Y_r$ and
 let $\mathcal{U}^i=\bigcup_{\alpha}\widetilde{\mathcal{U}_{\alpha}^i}$.

For $i\in\{-k,\cdots,-1\}$, let $\mathcal{W}^{i}=\mathcal{V}^{i}\cup_n\mathcal{U}^{i}$.
Since $\widetilde{\mathcal{U}_{\alpha}^i}$ is $n$-disjoint
and
$
d(X_{\alpha}\backslash Y_{r},X_{\alpha'}\backslash Y_{r})\geq r>n$
whenever $\alpha\neq\alpha'$,
$\mathcal{U}^i$ is $n$-disjoint and $R$-bounded.
Since $\mathcal{V}^{-1},\mathcal{V}^{-2},...,\mathcal{V}^{-k}$ are 5$R$-disjoint and uniformly bounded,
$\mathcal{W}^{i}$ is $n$-disjoint and uniformly bounded by Lemma \ref{r-saturated}.

For $i\in\{0,1,\cdots,m\}$,
since $\widetilde{\mathcal{U}_{\alpha}^i}$ is $r$-disjoint
and
$
d(X_{\alpha}\backslash Y_{r},X_{\alpha'}\backslash Y_{r})\geq r$
whenever $\alpha\neq\alpha'$,
$\mathcal{U}^i$ is $r$-disjoint and $D$-bounded.
\begin{itemize}
\item If $m\leq l$, let $\mathcal{W}^{i}=\mathcal{V}^{i}\cup_r\mathcal{U}^{i}$ for $i\in\{0,1,\cdots, m\}$
and $\mathcal{W}^{i}=\mathcal{V}^{i}$ for $i\in\{m+1,\cdots, l\}$.
Then by Lemma \ref{r-saturated}, $\mathcal{W}^i$ is $r$-disjoint and uniformly bounded and
$\bigcup_{i=-k}^{l}\mathcal{W}^{i}$ covers $X$, which means coasdim$(X)\leq fin+k$.
\item If $m> l$, let $\mathcal{W}^{i}=\mathcal{V}^{i}\cup_r\mathcal{U}^{i}$ for $i\in\{0,1,\cdots, l\}$
and $\mathcal{W}^{i}=\mathcal{U}^{i}$ for $i\in\{l+1,\cdots, m\}$. Similarly, $\mathcal{W}^i$ is $r$-disjoint and uniformly bounded and
$\bigcup_{i=-k}^{m}\mathcal{W}^{i}$ covers $X$, which implies coasdim$(X)\leq fin+k$.
\end{itemize}
\end{proof}

\begin{coro}
Let $X$ be a metric space with $X_{1}, X_{2}\subseteq X$. If coasdim$X_{i}\leq fin +k$ for some $k\in\NN$ and $i=1,2$,
then coasdim$(X_{1}\cup X_{2})\leq fin +k$.
\end{coro}
\begin{proof}
Apply Theorem \ref{finiteunion} to the family $\{X_{1}, X_{2}\}$ with $X_{2}=Y_{r}$ for every $r>0$.
\end{proof}

\begin{lem}\rm(Hurewicz Theorem in \cite{Hur})
\label{lem:Hurewicz}
Let $f : X \rightarrow Y$ be a Lipschitz map from a geodesic metric space $X$ to a metric space $Y$. Suppose that for every $R > 0$ the family $\{f^{-1}(B_R(y))\}_{y\in Y}$ satisfies the inequality asdim $\leq m$ uniformly, then asdim $X\leq$ asdim $Y+m$.
\end{lem}

\begin{thm}\label{quasihurewicztheorem}
Let $f : X \rightarrow Y$ be a 1-Lipschitz map from a geodesic metric space $X$ to a metric space $Y$. Suppose that for every $R > 0$, the family $\{f^{-1}(B_R(y))\}_{y\in Y}$ satisfies the inequality asdim $\leq m$ uniformly and coasdim$(Y)\leq fin+k$, then coasdim$(X)\leq fin+k(m + 1)$.
\end{thm}
\begin{proof}
Since coasdim$(Y)\leq fin+k$,
for every $n>0$, there exist~$R=R(n)>0$, $l=l(n)\in\NN$, $n$-disjoint and $R$-bounded families  $\mathcal{U}_1,...,\mathcal{U}_k$  such that
\[\text{asdim}(Y\setminus (\bigcup(\bigcup_{i=1}^k\mathcal{U}_i)))\leq l.
\]
For $i=1,2,\cdots,k$, since diam~$U\leq R$ for every $U\in \mathcal{U}_i$, there is a $y\in Y$ such that $U\subseteq B_R(y)$.
By the condition, the family $\{f^{-1}(B_R(y))\}_{y\in Y}$ satisfies the inequality asdim $\leq m$ uniformly.
So the family $\{f^{-1}(U)\}_{U\in \mathcal{U}_i}$ satisfies the inequality asdim $\leq m$ uniformly. i.e., there exists $S>0$ for $\forall i\in\{1,...,k\}$ such that for every $U\in \mathcal{U}_i$,
there are $n$-disjoint and $S$-bounded families $\mathcal{W}_{i}^0(U),\mathcal{W}_{i}^1(U),\cdots,\mathcal{W}_{i}^m(U)$ of subsets of $f^{-1}(U)$ such that
\[\bigcup_{j=0}^{m}\mathcal{W}_{i}^j(U) \text{~covers~} f^{-1}(U). \]
Let
$f^{-1}(\mathcal{U}_i)=\{f^{-1}(U)|U\in \mathcal{U}_i\}$. Since $f : X \rightarrow Y$ is a 1-Lipschitz map, $f^{-1}(\mathcal{U}_i)$ is $n$-disjoint.
For $j=0,1,2,\cdots,m$, let $\mathcal{W}_{i}^j=\bigcup_{U\in\mathcal{U}_{i}}\mathcal{W}_{i}^j(U)$. Since $f^{-1}(\mathcal{U}_i)$ is $n$-disjoint,
$\mathcal{W}_{i}^j$ is $n$-disjoint. And
\[\bigcup_{j=0}^{m}\mathcal{W}_{i}^j \text{~covers~} \bigcup_{U\in\mathcal{U}_{i}}f^{-1}(U). \]
Hence there are $n$-disjoint and $S$-bounded families $\{\mathcal{W}_{i}^j\}_{i=1,2,\cdots,k,j=0,1,\cdots,m}$ such that
$$\bigcup_{i=1,2,\cdots,k,j=0,1,\cdots,m}\mathcal{W}_{i}^j \text{~covers~} \bigcup(\bigcup_{i=1}^kf^{-1}(\mathcal{U}_i)).$$
Since asdim$(Y\setminus (\bigcup(\bigcup_{i=1}^k\mathcal{U}_i)))\leq l$, we have asdim$(f^{-1}(Y\setminus (\bigcup(\bigcup_{i=1}^k\mathcal{U}_i))))\leq l$
by Lemma \ref{lem:Hurewicz}. Then
$$\text{asdim}(X\setminus \bigcup_{1\leq i\leq k,0\leq j\leq m}\mathcal{W}_i^j)\leq \text{~asdim~}(f^{-1}(Y\setminus (\bigcup(\bigcup_{i=1}^k\mathcal{U}_i))))\leq l.$$
Therefore, coasdim$(X)\leq fin+k(m + 1)$.

\end{proof}

\end{section}

\begin{section}{ Cofinite asymptotic dimension is not stable under direct product}\

Several authors have studied the so-called product permanence properties. For example, asymptotic property C is one of them\cite{BellNag2017}. But unlike asymptotic property C, cofinite asymptotic dimension is not stable under product.\

Let $(X, d_X)$ and $(Y, d_Y)$ be metric spaces. We can define a metric on the
product $X\times Y$ by$$d_{X\times Y} ((x,y), (x',y')) = \text{max}\{d_X(x,x'),d_Y(y,y')\}.$$

\begin{prop}
\label{new1}
Let $X=\bigcup_{i=1}^{\infty}(2^{i}\mathbb{Z})^i$ be the metric space in Example 3.1, coasdim$(X\times X)< fin+fin$ is not true.
\end{prop}
\begin{proof}
Suppose that coasdim$(X\times X)< fin+fin$, then there exists $n\in\NN$, such that coasdim$(X\times X)\leq fin+n$.
For every $k>2^{n+1}$, there exist $R>0$, $m\in\ZZ^{+}$ and $k$-disjoint $R$-bounded families $\mathcal{U}_{-n},...,\mathcal{U}_{-1}$ such that asdim$(X\times X\setminus \bigcup(\bigcup_{i=-n}^{-1}\mathcal{U}_i))=m-1$.
By Corollary \ref{lem:dsup2}, for every $x\in (2^{m}\mathbb{Z})^{m}$,
$$(\{x\}\times ([-2^{n}R,2^{n}R]\bigcap 2^{n}\mathbb{Z})^n)\setminus \bigcup(\bigcup_{i=-n}^{-1}\mathcal{U}_i)\neq \emptyset.$$
i.e. $\exists (x,\Delta(x))\in(\{x\}\times ([-2^{n}R,2^{n}R]\bigcap 2^{n}\mathbb{Z})^n)\setminus \bigcup(\bigcup_{i=-n}^{-1}\mathcal{U}_i)$.\
Define a map
$$\delta:(2^{m}\mathbb{Z})^{m}\to (2^{m}\mathbb{Z})^{m}\times([-2^{n}R,2^{n}R]\bigcap 2^{n}\mathbb{Z})^n\setminus \bigcup(\bigcup_{i=-n}^{-1}\mathcal{U}_i)\text{  by }\delta(x)=(x,\Delta(x)).$$

Since $d(x,y)\leq d(\delta(x),\delta(y))\leq d(x,y)+2^{n+1}R$ for every $x,y\in (2^{m}\mathbb{Z})^{m}$, $\delta$ is a coarse embedding.

Therefore,
$$m= \text{asdim}(2^{m}\mathbb{Z})^{m})\leq \text{asdim}(2^{m}\mathbb{Z})^{m}\times([-2^{n}R,2^{n}R]\bigcap 2^{n}\mathbb{Z})^n\setminus \bigcup(\bigcup_{i=-n}^{-1}\mathcal{U}_i))$$
$$\leq \text{asdim}(X\times X\setminus \bigcup(\bigcup_{i=-n}^{-1}\mathcal{U}_i))=m-1$$
which is a contradiction.
\end{proof}

However the transfinite asymptotic dimension of $X\times X$ is not too big. In fact trasdim$(X\times X)=\omega$. In order to prove this result, we need some technical lemmas. The first one is motivated by \cite{Yama2015}.

\begin{lem}\rm
\label{lem1fornew2}
For every $ m,n,k\in \mathbb{N}$, $R>0$, there are $k$-disjoint uniformly bounded family $\mathcal{U}_0$ and $R$-disjoint uniformly bounded families $\mathcal{U}_1,...,\mathcal{U}_{m2^m}$ such that $\bigcup_{i=0}^{m2^m}\mathcal{U}_i$ covers $\mathbb{Z}^m\times (k\mathbb{Z})^n$.
\end{lem}

\begin{proof}
Let
\[ \mathcal{V}_0=\{[(2j-1)R,2jR):j\in \mathbb{Z}\} \text{~and~} \mathcal{V}_1=\{[2jR,(2j+1)R):j\in \mathbb{Z}\}.\]

For $l=1,...,n2^n$, let $S=R+k$ and
$$\mathcal{C}_l=\{[(2^nn(j-1)+l)S-R,(2^nnj+l)S-R-k):j\in \mathbb{Z}\}$$
$$\mathcal{D}_l=\{[(2^n nj+l)S-R-k,(2^n nj+l)S-R):j\in \mathbb{Z}\}$$
$$\mathcal{W}_l=\{\prod_{i=1}^nV_i:V_i\in \mathcal{V}_{\phi(l)_i},i\in\{1,...,n\}\}$$
in which $\phi$ is a bijection from $\{1,...,2^n\}$ to $\{0,1\}^n$.\\

Let
\[
\mathcal{U}_0=\{\prod_{i=1}^mC_i\times W:(\prod_{i=1}^mC_i,W)\in \bigcup_{l=1}^{2^n}(\mathcal{C}_l^m\times \mathcal{W}_l)\}
\]
\begin{multline*}
\mathcal{U}_{2^m(s-1)+t}=\{\prod_{i=1}^{s-1}V_i\times D_s\times \prod_{i=s+1}^{m}V_i \times W:(D_s,W)\in\bigcup_{l=1}^{2^n}(\mathcal{D}_l\times \mathcal{W}_l),\\V_i\in \mathcal{V}_{\rho(t)_i},i\in\{1,...,s-1,s+1,...,m\}\}
\end{multline*}
where $\rho$ is a bijection from $\{1,...,2^m\}$ to $\{0,1\}^m$ and $s\in\{1,...,m\}$, $t\in\{1,...,2^m\}$.

Each $\mathcal{C}_l$ is $k$-disjoint and the family $\{\mathbb{Z}^m\times W:W\in \bigcup_{l=1}^{2^n}\mathcal{W}_l\}$ is $k$-disjoint since $\bigcup_{l=1}^{2^n}\mathcal{W}_l$ is disjoint. Thus $\mathcal{U}_0$ is $k$-disjoint and uniformly bounded.

The families $\mathcal{V}_0,\mathcal{V}_1,\bigcup_{l=1}^{2^n}\mathcal{D}_l$ and $\{\mathbb{Z}^m\times W:W\in\mathcal{W}_l\}$ are $R$-disjoint. Thus $\mathcal{U}_j$ is $R$-disjoint and uniformly bounded for $j=1,2,\cdots, m2^m$.

We will prove that $\bigcup_{i=0}^{m2^m}\mathcal{U}_i$ cover $\mathbb{Z}^m\times (k\mathbb{Z})^n$.

Indeed, let $(x_i)_{i=1}^{m+n}\in(\mathbb{Z}^m\times (k\mathbb{Z})^n)\setminus \bigcup\mathcal{U}_0$, there exist $l\in \{1,...,2^n\}$ and $W\in \mathcal{W}_l$ such that $(x_i)_{i=m+1}^{m+n}\in W$. Since $(x_i)\notin \bigcup\mathcal{U}_0$, we have $(x_i)_{i=1}^m\notin \bigcup\{\prod_{i=1}^mC_i;C_i\in\mathcal{C}_l\}$. Then $\exists~ s\in\{1,...,m\}$ such that $x_s\notin \bigcup\mathcal{C}_l$. Since $\bigcup(\mathcal{C}_l\bigcup\mathcal{D}_l)=\mathbb{Z}$, there exists $D_s\in \mathcal{D}_l$ such that $x_s\in D_s$. Since $\bigcup(\mathcal{V}_0\bigcup\mathcal{V}_1)=\mathbb{Z}$, we may take $t\in\{1,...,2^m\}$ such that $(x_i)_{i=1}^m\in \prod_{i=1}^mV_i$ for some $V_i\in\mathcal{V}_{\phi(t)_i},i\in\{1,2,...,m\}$. Then we have
$$(x_i)_{i=1}^{m+n}\in\prod_{i=1}^{s-1}V_i\times D_s\times \prod_{i=s+1}^{m}V_i \times W\in\mathcal{U}_{2^m(s-1)+t}$$

\end{proof}

\begin{lem}\rm(\cite{Hur})
\label{lem-union-asdim}
Let $A$ and $B$ be subsets of a metric space $X$, then asdim $A\bigcup B\leq\text{max}\{\text{asdim}A,\text{asdim}B\}$.
\end{lem}

\begin{lem}\rm(\cite{Bell2011})
\label{lem-product-asdim}
Let $X$ and $Y$ be  metric spaces, then asdim $X\times Y\leq\text{asdim}X+\text{asdim}Y$.
\end{lem}

\begin{prop}
\label{new2}
Let $X=\bigcup_{i=1}^{\infty}(2^i\mathbb{Z})^i$ be the metric space in Example 3.1, then trasdim$(X\times X)=\omega$.
\end{prop}
\begin{proof}
By Proposition \ref{{iffforomega}}, it suffices for us to show for every $~k\in\NN$, there exists~$m=m(k)\in\NN,$ such that for every~ $n\in\NN$, there are uniformly bounded families $\mathcal{V}_0, \mathcal{V}_1, ...,\mathcal{V}_{m}$ satisfying $\mathcal{V}_0$ is $k$-disjoint, $\mathcal{V}_i$ is $n$-disjoint for $i=1,2,\cdots m$ and
$\bigcup_{i=0}^{m}\mathcal{V}_i$ covers $X\times X$. Moreover, $m\rightarrow \infty$ as $k\rightarrow \infty$.
Let $X\times X=Y_1\bigsqcup Y_2\bigsqcup Y_3\bigsqcup Y_4\bigsqcup Y_5\bigsqcup Y_6$ where
$$Y_1=\bigcup_{i=k+1}^{\infty}(2^i\mathbb{Z})^i\times \bigcup_{i=k+1}^{\infty}(2^i\mathbb{Z})^i,Y_2=\bigcup_{i=1}^k(2^i\mathbb{Z})^i\times \bigcup_{i=1}^k(2^i\mathbb{Z})^i,Y_3=\bigcup_{i=1}^k(2^i\mathbb{Z})^i\times \bigcup_{i=n+1}^{\infty}(2^i\mathbb{Z})^i,$$
$$Y_4=\bigcup_{i=n+1}^{\infty}(2^i\mathbb{Z})^i\times \bigcup_{i=1}^k(2^i\mathbb{Z})^i,Y_5=\bigcup_{i=1}^k(2^i\mathbb{Z})^i\times \bigcup_{i=k+1}^n(2^i\mathbb{Z})^iY_6=\bigcup_{i=k+1}^n(2^i\mathbb{Z})^i\times \bigcup_{i=1}^k(2^i\mathbb{Z})^i.$$

Let $\mathcal{U}_0^1=\{\{(x,y)\}:x,y\in\bigcup_{i=k+1}^{\infty}(2^{i}\mathbb{Z})^i\}$, then $\mathcal{U}_0^1$ is a $k$-disjoint uniformly bounded family which covers $Y_1$.

Let $X_1=\bigcup_{i=1}^k(2^{i}\mathbb{Z})^i$, since asdim$(2^{i}\mathbb{Z})^i=i$ for each $i\in\NN$, asdim$X_1=k$
by Lemma \ref{lem-union-asdim}. Then we obtain that asdim$Y_2=$asdim$(X_1\times X_1)\leq2k$ by Lemma \ref{lem-product-asdim}.
Therefore,
there exist $n$-disjoint and uniformly bounded families
$\mathcal{U}^2_0,...,\mathcal{U}^2_{2k}$ such that $\bigcup_{i=0}^{2k}\mathcal{U}^2_i$ cover $Y_2$.

Since asdim$X_1=k$, there exist $n$-disjoint and uniformly bounded families
$\mathcal{U}_0,...,\mathcal{U}_{k}$ such that $\bigcup_{i=0}^{k}\mathcal{U}_i$ cover $X_1$.
Let $\mathcal{U}_i^3=\{ U\times\{x\}:U\in \mathcal{U}_i,x\in\bigcup_{i=n+1}^{\infty}(2^{i}\mathbb{Z})^i\}$ for $i=0,...,k$, then $\mathcal{U}_i^3$ is $n$-disjoint uniformly bounded and $\bigcup_{i=0}^{k}\mathcal{U}_i^3$ covers $Y_3= X_1\times \bigcup_{i=n+1}^{\infty}(2^{i}\mathbb{Z})^i$.

Let $\mathcal{U}_i^4=\{ \{x\}\times U:U\in \mathcal{U}_i,x\in\bigcup_{i=n+1}^{\infty}(2^{i}\mathbb{Z})^i\}$ for $i=0,...,k$, then $\mathcal{U}_i^4$ is $n$-disjoint uniformly bounded and $\bigcup_{i=0}^{k}\mathcal{U}_i^4$ covers $Y_4=  (\bigcup_{i=n+1}^{\infty}(2^{i}\mathbb{Z})^i)\times X_1$.

There is an isometric map $\varphi:X_1\rightarrow (\mathbb{Z}^{k+1},d_{\text{max}})$ defined by
\[
 \forall x=(x_{1},x_{2},\cdots,x_{i})\in(2^{i}\mathbb{Z})^i,~\varphi(x)=(x_{1},x_{2},\cdots,x_{i},0,\cdots,0,\frac{i(i-1)}{2})\in\mathbb{Z}^{k+1}.
\]

and an isometric map $\phi:\bigcup_{i=k+1}^n(2^i\ZZ)^i \rightarrow ((2^k\mathbb{Z})^{n}\times\ZZ,d_{\text{max}})$ defined by
\[
 \forall x=(x_{1},x_{2},\cdots,x_{i})\in(2^{i}\mathbb{Z})^i,~\varphi(x)=(x_{1},x_{2},\cdots,x_{i},0,\cdots,0,\frac{i(i-1)}{2}).
\]
So $Y_3$ can be isometric embedded into $\mathbb{Z}^{k+2}\times (2^{k}\mathbb{Z})^{n}$. By Lemma \ref{lem1fornew2}, there are $k$-disjoint uniformly bounded family $\mathcal{V}_0$ and $n$-disjoint uniformly bounded families $\mathcal{V}_j$ for $j=1,...,2^{k+2}(k+2)$ such that $\bigcup_{j=0}^{2^{k+1}(k+1)}\mathcal{V}_j$ covers $\mathbb{Z}^{k+2}\times (2^{k}\mathbb{Z})^{n}$. Therefore, there are $k$-disjoint uniformly bounded family $\mathcal{U}_0^5$ and $n$-disjoint uniformly bounded families $\mathcal{U}_j^5$ for $j=1,...,2^{k+2}(k+2)$ such that $\bigcup_{j=0}^{2^{k+2}(k+2))}\mathcal{U}_j^5$ cover $Y_5$.

By the similar argument, there are $k$-disjoint uniformly bounded family $\mathcal{U}_0^6$ and $n$-disjoint uniformly bounded families $\mathcal{U}_j^6$ for $j=1,...,2^{k+2}(k+2)$ such that $\bigcup_{j=0}^{2^{k+2}(k+2)}\mathcal{U}_j^6$ cover $Y_6$.

Since $Y_1,Y_5,Y_6$ are $k$-disjoint, we obtain $k$-disjoint uniformly bounded family $\mathcal{U}_0=\mathcal{U}^1_0\bigcup\mathcal{U}^5_0\bigcup\mathcal{U}^6_0$ and $n$-disjoint uniformly bounded families
$$\mathcal{U}^2_j(j=0,...,2k),\mathcal{U}^3_j(j=0,...,k),\mathcal{U}^4_j(j=0,...,k),\mathcal{U}^5_j(j=1,...,2^{k+2}(k+2)),\mathcal{U}^6_j(j=1,...,2^{k+2}(k+2))$$
 s.t. $(\bigcup_{i,j}\mathcal{U}_{j}^i)\bigcup\mathcal{U}_0$ covers $X\times X$.
 \end{proof}

\begin{remark}
Theorem 2 in \cite{Radul2010} and Proposition \ref{new2} show that the space $X$ satisfies trasdim$(X)=\omega$ and trasdim$(X\times X)=\omega$. In fact, T. Yamauchi also gives us another space $\prod_{i=1}^{\infty}i\mathbb{Z}$ with the same property. By the same argument in \cite{Yama2015}, we know trasdim$(\prod_{i=1}^{\infty}i\mathbb{Z})=\omega$ and the space $\mathbb{Z}\times\mathbb{Z}\times2\mathbb{Z}\times2\mathbb{Z}\times3\mathbb{Z}\times3\mathbb{Z}\times ...$, with the same metric as $\prod_{i=1}^{\infty}i\mathbb{Z}$, has trasdim$(\mathbb{Z}\times\mathbb{Z}\times2\mathbb{Z}\times2\mathbb{Z}\times3\mathbb{Z}\times3\mathbb{Z}\times ...)=\omega$.
Since there is a isometric 'shift' bijection
$$\theta:\mathbb{Z}\times2\mathbb{Z}\times3\mathbb{Z}\times...\times \mathbb{Z}\times2\mathbb{Z}\times3\mathbb{Z}\times ...\to\mathbb{Z}\times\mathbb{Z}\times2\mathbb{Z}\times2\mathbb{Z}\times3\mathbb{Z}\times3\mathbb{Z}\times ...$$
defined by $\theta((x_i)\times(y_i))=(x_1,y_1,x_2,y_2,...)$, so trasdim$(\prod_{i=1}^{\infty}i\mathbb{Z}\times \prod_{i=1}^{\infty}i\mathbb{Z})=\omega$.\\
\end{remark}

\end{section}

\begin{section}{Transfinite asymptotic dimensions and shift union}\

In the present section we use the technique in \cite{Yama2015} to give a metric space $X$ with its transfinite asymptotic dimension $\leq\omega+1$ and we will give a negative answer to the Question 7.1 raised in \cite{BellNag2017}.

Let $d_{1}$ be the metric on $(\bigoplus_{i\in\mathbb{Z}}\mathbb{Z})\times\mathbb{Z}$ defined as following:
 \[\forall~x=((x_{i})_{i\in\mathbb{Z}},a), y=((y_{i})_{i\in\mathbb{Z}},b)\in(\bigoplus_{i\in\mathbb{Z}}\mathbb{Z})\times\mathbb{Z},
d_{1}(x,y)=\sum_{i\in\mathbb{Z}}|x_{i}-y_{i}|+|b-a|.
\]
Let
\[
X_{j}=\{((x_{i})_{i\in\mathbb{Z}},j)\in(\bigoplus_{i\in\mathbb{Z}}\mathbb{Z})\times\mathbb{Z}~~|~~
\begin{aligned}
&\text{~if~}i<j, \text{~then~} x_{i}=0; \\
&\text{~if~}i\geq j, \text{~then~}x_{i}\in(i-j+1)\mathbb{Z}
\end{aligned}
~\}
\]
and sh$\bigcup\bigoplus_{i=1}^{\infty}i\mathbb{Z}=\bigcup_{j\in\mathbb{Z}}X_j\subseteq((\bigoplus_{i\in\mathbb{Z}}\mathbb{Z})\times\mathbb{Z},d_{1})$. We call the metric space sh$\bigcup\bigoplus_{i=1}^{\infty}i\mathbb{Z}$ the \emph{shift union} of $\bigoplus_{i=1}^{\infty}i\mathbb{Z}$.

\begin{prop}
Let $X$ be the shift union sh$\bigcup\bigoplus_{i=1}^{\infty}i\mathbb{Z}$, then trasdim$X\leq\omega+1$.
\end{prop}

\begin{proof}

By Proposition \ref{{iffforomega}}, it suffices to show that for every $k\in\NN$, there exists~ $p=p(k)\in\NN,$ such that for every~ $m\in\NN$, there are uniformly bounded families $\mathcal{U}_0, \mathcal{U}_1, ...,\mathcal{U}_{p}$ satisfying $\mathcal{U}_0,\mathcal{U}_1$ are $k$-disjoint, $\mathcal{U}_i$ is $m$-disjoint for $i=2,\cdots, p$ and
$\bigcup_{i=0}^{p}\mathcal{U}_i$ covers $X$.

Let $Y_{n}=\bigcup_{i=2nk}^{2nk+2k-1}X_{i}$, then $X=\bigcup_{n\in\mathbb{Z}}Y_{n}$. It is easy to see that
$$ Y_{n}\subseteq Z_n\triangleq\{((x_{i})_{i\in\mathbb{Z}},a)~~|~~
\begin{aligned}
&a\in[2nk, 2nk+2k-1]\cap\mathbb{Z}; \\
&x_{i}=0 \qquad \text{if } i<2nk; \\
&x_{i}\in\mathbb{Z} \qquad \text{if } i\in[2nk,2nk+2k-1];\\
&x_{i}\in\bigcup_{j=2nk}^{2nk+2k-1}(i-j+1)\mathbb{Z} \qquad \text{if } i\geq2nk+2k.
\end{aligned}
~\}\subseteq(\bigoplus_{i\in\mathbb{Z}}\mathbb{Z})\times\mathbb{Z}.
$$
 Note that the family $\{Z_{2n}|~n\in\mathbb{Z}\}$ is $k$-disjoint.

Indeed, for every $n_{1},n_{2}\in\mathbb{Z}$ and $n_{1}<n_{2}$, $\forall~x=((x_{i})_{i\in\mathbb{Z}},a)\in Z_{2n_{1}}, y=((y_{i})_{i\in\mathbb{Z}},b)\in Z_{2n_{2}}$,
then $a\in[4n_{1}k, 4n_{1}k+2k-1]$ and $b\in[4n_{2}k, 4n_{2}k+2k-1]$. It follows that $$d(x,y)\geq d(a,b)=4n_{2}k-4n_{1}k-2k+1\geq 4k-2k+1>k.$$
Similarly, we obtain that the family $\{Z_{2n+1}|~n\in\mathbb{Z}\}$ is $k$-disjoint.

Let
$$\mathcal{V}_0=\{[(2n-1)m,2nm):n\in\mathbb{Z}\}\text{ and }\mathcal{V}_1=\{[2nm,(2n+1)m):n\in\mathbb{Z}\},$$
then $\mathcal{V}_0,\mathcal{V}_1$ are $m$-disjoint and $\mathcal{V}_0\cup\mathcal{V}_1$ covers $\mathbb{Z}$.

Let $S=k+m$, for $l=1,...,2^m$, let
$$\mathcal{C}_l=\{[(2^m(n-1)+l)2S-m,(2^mn+l)2S-m-k):n\in\mathbb{Z}\},$$
$$\mathcal{D}_l=\{[(2^mn+l)2S-m-k,(2^mn+l)2S-m):n\in\mathbb{Z}\},$$ and
$$\mathcal{W}_l=\{\prod_{i=1}^mV_i~|~V_i\in \mathcal{V}_{\phi(l)_i},i\in\{1,...,m\}\}$$
where $\phi$ is a bijection from $\{1,...,2^m\}$ to $\{0,1\}^m$.

Note that $\mathcal{C}_l$ is $k$-disjoint, $\mathcal{D}_l$ is $m$-disjoint and $\mathcal{D}_{l}\cup\mathcal{C}_{l}$ covers $\mathbb{Z}$.
Moreover, $\bigcup_{l=1}^{2^m}\mathcal{D}_l$ is $m$-disjoint.

Let
\begin{multline*}
\mathcal{U}_{0}^{n}=\{(\bigoplus_{i\leq2nk-1}\{0\}\times\prod_{i=2nk}^{2nk+3k-1}C_i\times W\times \{\widetilde{n}\},[2nk,2nk+2k-1]\cap\mathbb{Z})~|\\
\prod_{i=2nk}^{2nk+3k-1}C_i\times W\in\bigcup_{l=1}^{2^m}(\mathcal{C}_l^{3k}\times \mathcal{W}_l),\widetilde{n}\in \bigcup_{j=0}^{2k-1}\bigoplus_{i= k+m+2+j}^{\infty}(i\mathbb{Z}).\},
\end{multline*}
\begin{multline*}
\mathcal{U}_{2^{3k}(s-1)+t}^{n}=\{(\bigoplus_{i\leq2nk-1}\{0\}\times\prod_{i=2nk}^{2nk+s-1}V_i\times D_s\times \prod_{i=2nk+s+1}^{2nk+3k-1}V_i\times W\times \{\widetilde{n}\},[2nk,2nk+2k-1]\cap\mathbb{Z})~| \\V_{i+2nk-1}\in \mathcal{V}_{\rho(t)_i},i\in\{1,...,s-1,s+1,...,3k\},
(D_s,W)\in\bigcup_{l=1}^{2^m}(\mathcal{D}_l\times \mathcal{W}_l),\widetilde{n}\in \bigcup_{j=0}^{2k-1}\bigoplus_{i= k+m+2+j}^{\infty}(i\mathbb{Z}).\}
\end{multline*}
where $\rho$ is a bijection from $\{1,...,2^{3k}\}$ to $\{0,1\}^{3k}$, $s\in\{1,2,...,3k\}$ and $t\in\{1,2,3,...,2^{3k}\}$.

Note that each $\mathcal{C}_l$ is $k$-disjoint and the family
$\{\prod_{i\leq2nk-1}\{0\}\times\mathbb{Z}^{3k}\times W\times \prod_{i\geq 2nk+3k+m}\{n_i\}~|~W\in \bigcup_{l=1}^{2^m}\mathcal{W}_l\}$ is $k$-disjoint since $\bigcup_{l=1}^{2^m}\mathcal{W}_l$ is disjoint. Thus $\mathcal{U}_0^{n}$ is $k$-disjoint and uniformly bounded.

The families $\mathcal{V}_0,\mathcal{V}_1,\bigcup_{l=1}^{2^m}\mathcal{D}_l$ and $$\{\prod_{i\leq2nk-1}\{0\}\times\mathbb{Z}^{3k}\times W\times \prod_{i\geq 2nk+3k+m}\{n_i\}~|~W\in\mathcal{W}_l\}$$ are $m$-disjoint. Thus $\mathcal{U}_j^{n}$ is $m$-disjoint and uniformly bounded for $j=1,2,\cdots, (3k)2^{3k}$.

Let $\mathcal{U}_0\triangleq\bigcup_{n\in\mathbb{Z}}\mathcal{U}_0^{2n}$ and $\mathcal{U}_1\triangleq\bigcup_{n\in\mathbb{Z}}\mathcal{U}_0^{2n+1}$. Since $\{Z_{2n}|~n\in\mathbb{Z}\}$ and $\{Z_{2n+1}|~n\in\mathbb{Z}\}$ are $k$-disjoint, $\mathcal{U}_0$ and
 $\mathcal{U}_1$ are $k$-disjoint.
For $i=1,2,\cdots,(3k)2^{3k}$, let $\mathcal{U}_{2i}=\bigcup_{n\in\mathbb{Z}}\mathcal{U}_i^{2n}$ and $\mathcal{U}_{2i+1}=\bigcup_{n\in\mathbb{Z}}\mathcal{U}_i^{2n+1}$. Now we will show that
$\mathcal{U}_{i}$ is $m$-disjoint for $i=2,3,\cdots,(6k)2^{3k}+1.$

Indeed, for every $U,V\in\bigcup_{n\in\mathbb{Z}}\mathcal{U}_i^{2n}$ and $U\neq V$,
\begin{itemize}
\item if $U,V\in\mathcal{U}_i^{2n}$ for some $n\in\mathbb{Z}$, then $d(U,V)\geq m$.
\item if there exist $n_{1},n_{2}\in\mathbb{Z}$ and $n_{1}\neq n_{2}$ such that $U\in\mathcal{U}_i^{2n_{1}}$ and $V\in\mathcal{U}_i^{2n_{2}}$.

Assume that $n_{1}<n_{2}$,
for every $~x=((x_{i})_{i\in\mathbb{Z}},a)\in U\in\mathcal{U}_i^{2n_{1}}, y=((y_{i})_{i\in\mathbb{Z}},b)\in V\in\mathcal{U}_i^{2n_{2}}$,
there exists $s\in[4n_{1}k,4n_{1}k+3k-1]$, $x_{s}\in D_{s}\in\mathcal{D}_l$ for some $l\in\{1,2,\cdots, 2^{m}\}$.
Since $s\leq4n_{1}k+3k-1<4n_{2}k$,  $y_{s}=0$.
It is not difficult to see that for every $l\in\{1,...,2^m\}$, $[-m,m]\subseteq\bigcup\mathcal{C}_l$.
Therefore, $[-m,m]\cap D_{s}=\emptyset$.
It follows that $d(x,y)\geq |x_{s}-y_{s}|=|x_{s}|> m.$
So $d(U,V)\geq m$.
\end{itemize}
Therefore, for $i=1,2,\cdots,(3k)2^{3k}$, $\mathcal{U}_{2i}=\bigcup_{n\in\mathbb{Z}}\mathcal{U}_i^{2n}$ is $m$-disjoint.
Similarly, we can also prove that $\mathcal{U}_{2i+1}=\bigcup_{n\in\mathbb{Z}}\mathcal{U}_i^{2n+1}$ is $m$-disjoint for $i=1,2,\cdots,(3k)2^{3k}$.

We will prove that for every $n\in\mathbb{Z}$, $\bigcup_{i=0}^{(3k)2^{3k}}\mathcal{U}_i^n$ cover $Z_{n}$.

Indeed, let $x=((x_{i})_{i\in\mathbb{Z}},a)\in Z_{n}\setminus \bigcup\mathcal{U}_0^{n}$,
then  there exists $l\in \{1,...,2^m\}$ and $W\in \mathcal{W}_l$ such that $(x_i)_{i=2nk+3k}^{2nk+3k+m-1}\in W$. Since $x\notin \bigcup\mathcal{U}_0$, we have $(x_i)_{i=2nk}^{2nk+3k-1}\notin \bigcup\{\prod_{i=2nk}^{2nk+3k-1}C_i~|~C_i\in\mathcal{C}_l\}$. Then there exists $s\in\{1,2,...,3k\}$ such that $x_{s+2nk-1}\notin \bigcup\mathcal{C}_l$. Because $\bigcup(\mathcal{C}_l\bigcup\mathcal{D}_l)=\mathbb{Z}$, there exists $D_s\in \mathcal{D}_l$ such that $x_{s+2nk-1}\in D_s$. Since $\bigcup(\mathcal{V}_0\bigcup\mathcal{V}_1)=\mathbb{Z}$, we may take $t\in\{1,...,2^{3k}\}$ such that $(x_i)_{i=2nk}^{2nk+3k-1}\in \prod_{i=2nk}^{2nk+3k-1}V_i$, where $V_{i+2nk-1}\in\mathcal{V}_{\rho(t)_i},i\in\{1,2,...,3k\}$. So
$x\in\mathcal{U}_{2^{3k}(s-1)+t}^n.$

Since $\bigcup_{i=0}^{(3k)2^{3k}}\mathcal{U}_i^n$ covers $Z_n$ and $Y_n\subseteq Z_n$, $\bigcup_{n\in\mathbb{Z}}\bigcup_{i=0}^{(3k)2^{3k}}\mathcal{U}_i^n$ covers $X=\bigcup_{n\in\mathbb{Z}}Y_n$. Therefore, $\bigcup_{i=0}^{(6k)2^{3k}}\mathcal{U}_i$ covers $X$.
\end{proof}

\textbf{Question}: For $X=...\times 3\mathbb{Z}\times 2\mathbb{Z}\times \mathbb{Z}\times 2\mathbb{Z}\times 3\mathbb{Z}\times ...$, whether trasdim$(\text{sh}\bigcup X)=\omega$? Or whether $\text{sh}\bigcup X$ has asymptotic property C?\\

Finally, we will give a negative answer to the Question 7.1 raised in \cite{BellNag2017}.

\textbf{Question}\cite{BellNag2017}: Let $f: X\rightarrow Y$ be a uniformly expansive map between metric spaces. Assume that $Y$ has asymptotic property C
and $f^{-1}(A)$ has asymptotic property C for every bounded subset $A\subseteq Y$. Does $X$ have to have asymptotic property C?

\begin{exa}
For every $k\in\mathbb{Z}^{+}$, let
\[
X_k=\{(x_{i})_{i=1}^{\infty}\in\bigoplus_{i=1}^{\infty}\mathbb{Z}~|~x_{i}=0 ~~\text{for every~}i>k\}.
 \]
Let $X$ be a subspace of $((\bigoplus_{i=1}^{\infty}\mathbb{Z})\times\mathbb{Z},d_{1})$  defined by $$X=\{((a^k_i)_{i=1}^{\infty},k)\in(\bigoplus_{i=1}^{\infty}\mathbb{Z})\times\mathbb{Z}|~(a^k_i)_{i=1}^{\infty}\in X_{k},k\in\mathbb{Z}^{+}\}.$$
Define a map $f:X\to \mathbb{Z}$ by
$$f(((a^k_i)_{i=1}^{\infty},k))=k,~~\text{ for every }((a^k_i)_{i=1}^{\infty},k)\in X.$$
It is easy to see that $f$ is uniformly expansive and $\mathbb{Z}$ has asymptotic dimension 1. For any bounded subset $A\subseteq\mathbb{Z}$, there exists $n\in \mathbb{Z}$ such that $A\subset[1,n]\bigcap\mathbb{Z}$, so we have $f^{-1}(A)\subset  X_{n}\times[1,n]$, which has asymptotic dimension no more than $n$ and hence has asymptotic property C. Since $\mathbb{Z}^{m}\cong(X_{m},m)\subset X$ for every $m\in \mathbb{Z}$, $X$ does not have asymptotic property C.\\
\end{exa}
\end{section}

{\bf Acknowledgments.} The author wish to thank the reviewers for careful
reading and valuable comments. This work was supported by NSFC grant of P.R. China (No.11301224,11326104). And the authors are
grateful to Moscow State University where part of this paper has been written.


\providecommand{\bysame}{\leavevmode\hbox to3em{\hrulefill}\thinspace}
\providecommand{\MR}{\relax\ifhmode\unskip\space\fi MR }
\providecommand{\MRhref}[2]{%
  \href{http://www.ams.org/mathscinet-getitem?mr=#1}{#2}
}
\providecommand{\href}[2]{#2}

\end{document}